\documentclass{amsart}
\usepackage{amssymb}
\usepackage{amscd}
\usepackage{multicol}
\usepackage{tikz}
\usepackage{graphics}
\usepackage{shuffle}
\usepackage[OT2,T1]{fontenc}
\usepackage[all]{xy}
\usepackage{extarrows}
\usepackage{comment}
\DeclareSymbolFont{cyrletters}{OT2}{wncyr}{m}{n}
\DeclareMathSymbol{\Sha}{\mathalpha}{cyrletters}{"58}

\usepackage{ulem}

\title[Renormalized values of shuffle type and of harmonic type]
{Relationship between renormalized values of shuffle type and of harmonic type of multiple zeta functions\\
}
\author{Nao Komiyama}

\address{Graduate School of Mathematics, Nagoya University, 
Furo-cho, Chikusa-ku, Nagoya 464-8602 Japan }
\email{m15027u@math.nagoya-u.ac.jp}
\thanks{}
\date{March 31, 2021}

\newtheorem{thm}{Theorem}[section]
\newtheorem{lem}[thm]{Lemma}
\newtheorem{cor}[thm]{Corollary}
\newtheorem{prop}[thm]{Proposition}  

\theoremstyle{remark}
        
\subjclass[2010]{}

\numberwithin{equation}{section}
\theoremstyle{definition}
\newtheorem{definition}[thm]{Definition}
\newtheorem{remark}[thm]{Remark}

\newtheorem{example}[thm]{Examples}

\newtheorem{prob}[thm]{Problem}

\newcommand{\N}{{\mathbb N}}
\newcommand{\Z}{{\mathbb Z}}
\newcommand{\Q}{{\mathbb Q}}

\newcommand{\C}{{\mathbb C}}

\newcommand{\QSh}[3]{{\rm QSh}\binom{#1;#2}{#3}}

\newcommand{\emp}{{\bf 1}}

\newcommand{\EMS}{{\rm EMS}}
\newcommand{\MP}{{\rm MP}}
\newcommand{\GZ}{{\rm GZ}}

\begin{document}
\bibliographystyle{amsalpha+}
\maketitle

\begin{abstract}
In this paper, we settle the problem posed by Singer which is on a comparison problem between the renormalized values of shuffle type and harmonic type of multiple zeta functions.
\end{abstract}

\tableofcontents
\setcounter{section}{-1}
\section{Introduction}
The {\it multiple zeta function} (MZF for short) is defined by
\begin{equation*}
	\zeta(s_1,\dots,s_r):=\sum_{0<m_1<\cdots<m_r}\frac{1}{m_1^{k_1}\cdots m_r^{k_r}}
\end{equation*}
and converges absolutely in the region
$$
\left\{ (s_1,\dots,s_r)\in\C^r\ |\ \Re(s_{r-k+1}+\cdots+s_r)>k ,\ 1\leq k \leq r \right\}.
$$
Especially, $\zeta(s_1,\dots,s_r)$ is called the {\it multiple zeta value} (MZV for short) for $s_1,\dots,s_{r-1}\in\Z_{\geq1}$ and $s_r\in\Z_{\geq2}$.
In \cite{AET}, it is shown that MZF can be meromorphically continued to $\C^r$, and all singularities of MZF are explicitly determined as
\begin{align}\label{all pole of MZF}
	&s_r=1, \nonumber\\
	&s_{r-1}+s_r=2,1,0,-2,-4,\dots, \\
	&s_{r-k+1}+\cdots+s_r=k-n\quad (3\leq k\leq r,\ n\in\mathbb{N}_0). \nonumber
\end{align}
Because almost all integer points are located in the above singularities, the special values of MZFs there are not determined.
As one of the way which gives a nice definition of the special values of MZFs at integer points, the following problem is proposed.
\begin{prob}[Problem \ref{prob:reno prob}, Renormalization problem of MZVs]
Extend MZVs to all integer points such that
\begin{enumerate}
\renewcommand{\labelenumi}{(\Alph{enumi}).}
	\item the values coincide with the special values of analytic continuation of MZFs,
	\item the harmonic relations or the shuffle relations are preserved.
\end{enumerate}
\end{prob}
In connection with the above problem, Guo and Zhang (\cite{GZ}), Manchon and Paycha (\cite{MP}) and Ebrahimi-Fard, Manchon and Singer (\cite{EMS2}) independently give the renormalized values of harmonic type of MZFs at negative integers.
Moreover, in \cite{EMSZ}, it is shown that there are infinitely many solutions of the above problem for the harmonic relations (see Theorem \ref{thm: EMSZ}).
While, the renormalized values of shuffle type of MZFs at non-positive integers are given only by Ebrahimi-Fard, Manchon and Singer (\cite{EMS1}).
In the final line of \cite{Singer}, Singer mention the following problem on a relationship between the renormalized values of harmonic type and shuffle type.

\begin{prob}[Problem \ref{prob: Singer problem}]
Which renormalized value of harmonic type has an explicit relationship with the renormalized values $\zeta_{\EMS}(-k_1,\dots,-k_r)$ (defined in Definition \ref{def:renormalized values})?
\end{prob}
Our main result in this paper is the following theorem which settle the above problem.
\begin{thm}[Theorem \ref{thm:EMS=GZ,MP}]
For $r\geq1$, we have
\begin{equation}\label{eqn:EMS=GZ,MP in intro}
Z_{\EMS}(t_1,\dots,t_r)
	=\sum_{i=1}^r\sum_{\sigma\in\mathcal P(r,i)}
	Z_*\left(u_{\sigma^{-1}(1)},\dots,u_{\sigma^{-1}(i)}\right).
\end{equation}
Here, $Z_{\EMS}(t_1,\dots,t_r)$ is defined by \eqref{eqn:gen fun of zetaEMS} and $Z_*(t_1,\dots,t_r)$ is defined in Definition \ref{eqn:gen fun of ren val of har type}, and for $r,i\in\N$ with $i\leq r$, the set $\mathcal P(r,i)$ (see \S \ref{sub:Relationship between renormalized values} for detail) is defined by
$$
\mathcal P(r,i):=\left\{\sigma:\{1,\dots,r\}\twoheadrightarrow\{1,\dots,i\}\right\},
$$
and, for $\sigma\in\mathcal P(r,i)$, the symbol $u_{\sigma^{-1}(k)}$ is defined by
$$
u_{\sigma^{-1}(k)}:=\sum_{n\in\sigma^{-1}(k)}u_n
$$
for $u_i:=t_i+\cdots+t_r$ ($1\leq i\leq r$).
\end{thm}
We denote the renormalized values introduced in \cite{GZ} and \cite{MP} by 
\begin{align*}
	\zeta_{\GZ}(-k_1,\dots,-k_r) \mbox{\quad and \quad} \zeta_{\MP}(-k_1,\dots,-k_r),
\end{align*}
for $k_1,\dots,k_r\in \Z_{\leq0}$, and define their generating functions by
\begin{align*}
	&Z_{\GZ}(t_1,\dots,t_r):=\sum_{k_1,\dots,k_r=0}^{\infty}\frac{(-t_1)^{k_1}\cdots(-t_r)^{k_r}}{k_1!\cdots k_r!}\zeta_{\GZ}(-k_1,\dots,-k_r), \\
	&Z_{\MP}(t_1,\dots,t_r):=\sum_{k_1,\dots,k_r=0}^{\infty}\frac{(-t_1)^{k_1}\cdots(-t_r)^{k_r}}{k_1!\cdots k_r!}\zeta_{\MP}(-k_1,\dots,-k_r).
\end{align*}
\begin{cor}[Theorem \ref{cor:EMS=GZ,MP}]
The equation \eqref{eqn:EMS=GZ,MP in intro} holds for $Z_*=Z_{\GZ}$ and $Z_{\MP}$.
Hence, the renormalized values $\zeta_{\EMS}(-k_1,\dots,-k_r)$ can be represented by a finite linear combination of either $\zeta_{\GZ}(-k_1,\dots,-k_r)$ or $\zeta_{\GZ}(-k_1,\dots,-k_r)$.
\end{cor}

Our plan in this paper is the following.
In \S \ref{Renormalized values of shuffle type}, we recall the definition of the renormalized values introduced in \cite{EMS1} and recall some properties.
In \S \ref{Renormalized values related with harmonic relations}, we treat the problem posed by Singer (\cite{Singer}) which is on a comparison problem between the renormalized values of shuffle type and harmonic type.
In \S \ref{sub:Relationship between renormalized values}, we settle the problem by giving a universal presentation of the renormalized values of \cite{EMS1} as finite linear combinations of any renormalized values of harmonic type (Theorem \ref{thm:EMS=GZ,MP}).
\section{Renormalized values of shuffle type}\label{Renormalized values of shuffle type}

In this section, we recall the definition of the renormalized values introduced in \cite{EMS1}.
Let $L:=\{d,y\}$ and put $L^*$ to be the non-commutative free monoid generated by $L$.
We consider the two variables non-commutative polynomial $\Q$-algebra $\Q\langle L\rangle$ with the empty word $\emp$.
We define the product $\shuffle_0: \Q\langle L\rangle^{\otimes2} \rightarrow \Q\langle L\rangle$ by $\emp\shuffle_0 w:=w\shuffle_0\emp:=w$ and
\begin{align*}
	yu\shuffle_0v &:= u\shuffle_0yv:=y(u\shuffle_0v), \\
	du\shuffle_0dv &:= d(u\shuffle_0dv) - u\shuffle_0d^2v,
\end{align*}
for any words $w,u,v\in L^*$.
Then $(\mathbb{Q}\langle L\rangle, \shuffle_0)$ forms a unitary, nonassociative, noncommutative $\mathbb{Q}$-algebra.
We define
\begin{equation*}
	\mathcal{T}_-:= \langle \{wd\ | \ w \in L^*\}\rangle_{\mathbb{Q}},
\end{equation*}
that is, to be the $\Q$-linear subspace of $\mathbb{Q}\langle L\rangle$ linearly generated by words ending in $d$.
We define
\begin{equation*}
	\mathcal{L}_-:= \langle d^k\{d(u \ \shuffle_0 \ v) -d u \ \shuffle_0 \ v - u \ \shuffle_0 \ dv\} \ |\ k \in \mathbb{N}_0,\ u,v \in L^*y \cup \{{\bf 1}\} \ \rangle_{(\mathbb{Q}\langle L\rangle,\shuffle_0)},
\end{equation*}
that is, to be the two-sided ideal of $(\mathbb{Q}\langle L\rangle,\shuffle_0)$ algebraically generated by the above elements.
We consider the $\mathbb{Q}$-linear subspace
\begin{equation*}
	\mathcal{S}_-:=\mathcal{T}_-+\mathcal{L}_-
\end{equation*}
 of $\mathbb{Q}\langle L\rangle$ generated by $\mathcal{L}_-$ and $\mathcal{T}_-$. This $\mathcal{S}_-$ also forms a two-sided ideal of $(\mathbb{Q}\langle L\rangle,\shuffle_0)$.
We put the quotient
\begin{equation}\label{eqn:def of H0}
	\mathcal{H}_0 := \mathbb{Q}\langle L \rangle /\mathcal{S}_-.
\end{equation}
Then $\mathcal{H}_0$ forms a connected, filtered, commutative and cocommutative Hopf algebra (cf. \cite[\S3.3.6]{EMS1}), whose product is equal to $\shuffle_0$ and whose coproduct is given by 
\begin{equation*}
	\Delta_0(w) := \sum_{\substack{S \subset [n]\\ S:{\rm admissible}}}w_S \otimes w_{\overline{S}},
\end{equation*}
for $w \in L^*y$. In the summation, $S$ may be empty. we put $n:={\rm wt}(w)$, $[n] := \{1,\dots,n\}$ and $\overline{S} := [n]\setminus S$. For $w:=x_1\cdots x_n\ (x_i\in L^*,\ i=1,\dots,n)$ and $S:=\{i_1,\dots,i_k\}$ with $1\leq i_1<\cdots<i_k\leq n$, we define $w_S:=x_{i_1}\cdots x_{i_k}$. We call the set $S$ {\it admissible} if both $w_S, w_{\overline{S}} \in L^*y \cup \{{\bf 1}\}$. See \cite[\S 3.3.8]{EMS1} for combinatorial method using polygons to compute $\Delta_0(w)$. We define the $\Q$-linear map $\tilde{\Delta}_0:\mathcal{H}_0\rightarrow\mathcal{H}_0\otimes\mathcal{H}_0$ by
\begin{equation}\label{eqn:reduced coproduct}
	\tilde{\Delta}_0(w) := \Delta_0(w)-1\otimes w-w\otimes 1 \quad (w \in Y),
\end{equation}
and we call $\tilde{\Delta}_0$ the {\it reduced coproduct}.

Let $\mathcal{A}:=\mathbb{Q}[\frac{1}{z},z]]:=\mathbb{Q}[[z]][\frac{1}{z}]$ be the algebra consisting of all Laurent series. And we decompose it as $\mathcal{A}=\mathcal{A}_-\oplus\mathcal{A}_+$ where $\mathcal{A}_-:={\frac{1}{z}\mathbb{Q}[\frac{1}{z}]}$ and $\mathcal{A}_+:=\mathbb{Q}[[z]]$.
Let $\mathcal{H}$ be a Hopf algebra over $\mathbb{Q}$ and $\mathcal{L}(\mathcal{H},\mathcal{A})$ be the set of $\mathbb{Q}$-linear maps from $\mathcal{H}$ to $\mathcal{A}$. We define the {\it convolution} $\phi \star \psi \in \mathcal{L}(\mathcal{H},\mathcal{A})$ by
\begin{equation*}
	\phi \star \psi:=m_{\mathcal{A}}\circ(\phi\otimes\psi)\circ\Delta_{\mathcal{H}}
\end{equation*}
for $\mathbb{Q}$-linear maps $\phi\ \mbox{and}\ \psi \in \mathcal{L}(\mathcal{H},\mathcal{A})$.
Let $\mathcal{H}$ be a Hopf algebra over $\mathbb{Q}$ and $\mathcal{A}$ be a $\mathbb{Q}$-algebra. The subset
\begin{equation*}
	G(\mathcal{H},\mathcal{A}) := \{\phi \in \mathcal{L}(\mathcal{H},\mathcal{A})\ |\ \phi({\bf 1}_{\mathcal H})={\bf 1}_{\mathcal{A}}\}
\end{equation*}t
endowed with the above convolution product $\star$ forms a group. The unit is given by a map $e =u_{\mathcal{A}}\circ\varepsilon_{\mathcal{H}}$.
The following theorem is the fundamental tool of Connes and Kreimer (\cite{CK}) in the renormalization procedure of perturbative quantum field theory.
\begin{thm}[\cite{CK}, \cite{EMS1}, \cite{Man}: {\bf algebraic Birkhoff decomposition}]\label{thm:algebraic Birkhoff decomposition}
For $\phi \in G(\mathcal{H},\mathcal{A})$, there are unique linear maps $\phi_+:\mathcal{H}\rightarrow \mathcal{A}_+$ and $\phi_-:\mathcal{H}\rightarrow \mathbb{Q}\oplus\mathcal{A}_-$ with $\phi_-({\bf 1})=1\in\mathbb{Q}$ such that
\begin{equation*}
	\phi=\phi_-^{\star-1} \star \phi_+.
\end{equation*}
Moreover the maps $\phi_-$ and $\phi_+$ are algebra homomorphisms if $\phi$ is an algebra homomorphism.
\end{thm}
We define the $\mathbb{Q}$-linear map $\phi:\mathcal{H}_0 \rightarrow \mathcal{A}$ by $\phi({\bf 1}):=1$ and for $k_1,\dots,k_r \in \mathbb{N}_0$,
\begin{equation*}
	d^{k_1}y\cdots d^{k_r}y \mapsto \phi(d^{k_1}y\cdots d^{k_r}y)(z) := \partial^{k_1}_z\left(x\partial^{k_2}_z\right)\cdots\left(x\partial^{k_r}_z\right)\left(x(z)\right)
\end{equation*}
where $x:=x(z) := \frac{e^z}{1-e^z} \in \mathcal{A}$ and $\partial_z$ is the derivative by $z$.
Then this map $\phi$ forms an algebra homomorphism (see \cite[Lemma 4.2]{EMS1}).
By applying Theorem \ref{thm:algebraic Birkhoff decomposition} to this map $\phi$, we get the algebra homomorphism $\phi_+:\mathcal H_0 \rightarrow \Q[[z]]$.
\begin{definition}[{\cite[\S 4.2]{EMS1}}]\label{def:renormalized values}
The {\it renormalized value} \footnote{If we follow the notations of \cite{EMS1}, it should be denoted by $\zeta_+(-k_r,\dots,-k_1)$.} $\zeta_{\scalebox{0.5}{\rm EMS}}(-k_1,\dots, -k_r)$ is defined by
\begin{equation*}
	\zeta_{\scalebox{0.5}{\rm EMS}}(-k_1,\dots, -k_r)
	:= \lim_{z \rightarrow 0}\phi_+(d^{k_r}y\cdots d^{k_1}y)(z)
\end{equation*}
for $k_1,\dots,k_r \in \mathbb{N}_0$.
\end{definition}
We have the following proposition.
\begin{prop}[{\cite[Proposition 3.3]{Comy1}}]\label{prop:simple recurrence formula of zetaEMS}
For $r \in \mathbb{N}_{\geq 2}$ and $k_1, \dots, k_r \in \mathbb{N}_0$, we have
\begin{equation*}\label{eqn:simple recurrence formula of zetaEMS}
	\zeta_{\scalebox{0.5}{\rm EMS}}(-k_1, \dots, -k_r)
	=\sum_{\substack{
	i + j=k_r \\
	i,j\geq0}}	\binom{k_r}{i}
	\zeta_{\scalebox{0.5}{\rm EMS}}(-i) \zeta_{\scalebox{0.5}{\rm EMS}}(-k_1,\dots, -k_{r-1}-j),
\end{equation*}
\end{prop}
We define the generating function $Z_{\scalebox{0.5}{\rm EMS}}(t_1,\dots,t_r)$ of the renormalized values $\zeta_{\scalebox{0.5}{\rm EMS}}(-k_1,\dots, -k_r)$ by
\begin{equation}\label{eqn:gen fun of zetaEMS}
Z_{\scalebox{0.5}{\rm EMS}}(t_1,\dots,t_r)
:=\sum_{k_1,\dots,k_r\geq0}
\frac{(-t_1)^{k_1}\cdots (-t_r)^{k_r}}{k_1!\cdots k_r!}\zeta_{\scalebox{0.5}{\rm EMS}}(-k_1,\dots, -k_r).
\end{equation}
In \cite{Comy1}, the explicit formula of this generating function is given as follow.
\begin{thm}[{\cite[Corollary 3.9]{Comy1}}]
We have
\begin{equation*}
	Z_{\scalebox{0.5}{\rm EMS}}(t_1,\dots,t_r)
	=\prod_{i=1}^r
	\frac{(t_i+\cdots+t_r)-(e^{t_i+\cdots+t_r}-1)}{(t_i+\cdots+t_r)(e^{t_i+\cdots+t_r}-1)}.
\end{equation*}
\end{thm}

\begin{remark}
By the above theorem, we have
\begin{equation}\label{eqn:recurrence formula of gen fun}
	Z_{\scalebox{0.5}{\rm EMS}}(t_1,\dots,t_r)
	=Z_{\scalebox{0.5}{\rm EMS}}(t_r)Z_{\scalebox{0.5}{\rm EMS}}(t_{r-1}+t_r)
	\cdots Z_{\scalebox{0.5}{\rm EMS}}(t_1+\cdots+t_r).
\end{equation}
\end{remark}

In \S \ref{sub:Relationship between renormalized values}, we will give an explicit relationship between  $\zeta_{\scalebox{0.5}{\rm EMS}}(-k_1,\dots, -k_r)$ and renormalized values of harmonic type introduced in \S \ref{Renormalized values related with harmonic relations}.

\section{Renormalized values of harmonic type}\label{Renormalized values related with harmonic relations}

In this section, we reformulate a certain problem between renormalized values posed in the final line of [S] as Problem \ref{prob: Singer problem}.
We start with the following problem.
\begin{prob}[Renormalization problem of MZVs (cf. {\cite[Problem 1]{Singer}})]\label{prob:reno prob}
Extend MZVs to all integer points such that
\begin{enumerate}
\renewcommand{\labelenumi}{(\Alph{enumi}).}
	\item the values coincide with the special values of analytic continuation of MZFs,
	\item the harmonic relations or the shuffle relations are preserved.
\end{enumerate}
\end{prob}
Based on \cite{EMSZ}, we recall the solutions of this problem.
Let $\mathcal H:=\Q\langle z_k\ |\ k\in\Z \rangle$ be the non-commutative polynomial algebra with the empty word $\emp$ generated by the letters $z_k$.
Then $(\mathcal H, *, \Delta)$ is a Hopf algebra.
Here, the product $*$ is the harmonic product, which is given by $w*\emp:=\emp*w:=w$ and 
\begin{equation}\label{eqn:def of harmonic}
z_kw*z_lw':=z_k(w*z_lw')+z_l(z_kw*w')+z_{k+l}(w*w'),
\end{equation}
for $k,l\in\Z$ and words $w,w'$ in $\mathcal H$, and the coproduct $\Delta$ is the deconcatenation coproduct.
\begin{definition}[{\cite[Definition 4.2]{EMSZ}}]\label{def: non-singular}
We call a word $w=z_{k_1}\cdots z_{k_r}$ in $\mathcal H$ {\it non-singular} if all of the following conditions hold:
\begin{align*}
	&k_r \neq 1, \\
	&k_{r-1}+k_r \neq 2,1,0,-2,-4,\dots, \\
	&k_{r-i+1}+\cdots+k_r \neq i-n\quad (3\leq i\leq r,\ n\in\mathbb{N}_0). 
\end{align*}
We denote $N \subset \mathcal H$ to be the $\C$-vector space spanned by all non-singular words.
\end{definition}
We define the $\C$-linear map $\zeta^*: N \rightarrow \C$ by 
$$
\zeta^*(z_{k_1}\cdots z_{k_r}):=\zeta(k_1, \dots, k_r),
$$
for $z_{k_1}\cdots z_{k_r}\in N$, where the right hand side is the special values of analytic continuation of MZF.
We put $G_\C$ to be the set of all algebra homomorphisms from $\mathcal H$ to $\C$, and put the convolution product $\star:G_\C \otimes G_\C \rightarrow G_\C$ by
$$
f \star g:=m \circ (f\otimes g) \circ \Delta,
$$
for any $f,g\in G_\C$, where $m$ is the ordinary product of $\C$.
Then $(G_\C, \star)$ forms a group.
\begin{definition}[{\cite[Definition 4.5]{EMSZ}}]
We define the set $X_{\C,\zeta^*}$ of all solutions of Problem \ref{prob:reno prob} for harmonic relations by
$$
X_{\C,\zeta^*} := \{ \phi\in G_\C\ |\ \phi|_N=\zeta^* \},
$$
and we define the set $T_\C$ called the {\it renormalization group} by
$$
T_{\C} := \{ \phi\in G_\C\ |\ \phi|_N=0 \}.
$$
\end{definition}
\begin{thm}[{\cite{EMSZ} (cf. \cite[Theorem 16]{Singer})}]\label{thm: EMSZ}
We have:
\begin{enumerate}
\renewcommand{\labelenumi}{{\rm (\alph{enumi}).}}
	\item The set $T_\C$ forms a subgroup of $(G_\C,\star)$.
	\item The left group action
	\begin{align*}
		T_\C \times X_{\C,\zeta^*} & \longrightarrow X_{\C,\zeta^*}, \\
		(\alpha,\phi) & \longmapsto \alpha \star \phi
	\end{align*}
	is free and transitive.
	\item The cardinality of the set $X_{\C,\zeta^*}$ is infinite.
\end{enumerate}
\end{thm}
We put $\mathcal H_{\leq0}$ to be the subalgebra of $\mathcal H$ generated by $\{ z_k\ |\ k\in\Z_{\leq0} \}$.
Then it is immediate that $\mathcal H_{\leq0}$ forms a Hopf subalgebra of $(\mathcal H, *, \Delta)$.
We define $G_{\leq0}$ to be the set of all restrictions of elements in $G_\C$ to $\mathcal H_{\leq0}$.
\begin{definition}\label{def:ren val of har type}
We define the set $X_{\leq0}$ of {\it renormalized values (at non-positive integer points) of harmonic type} by
$$
X_{\leq0} := \{ \phi\in G_{\leq0}\ |\ \phi|_{N \cap \mathcal H_{\leq0}}=\zeta^* \}.
$$
\end{definition}
By using this, we reformulate the problem which is mentioned in \cite{Singer} as follows.
\begin{prob}[{The final line of \cite{Singer}}]\label{prob: Singer problem}
Which renormalized value of harmonic type have an explicit relationship with the renormalized values $\zeta_{\EMS}(-k_1,\dots,-k_r)$ (defined in Definition \ref{def:renormalized values})?
\end{prob}

\begin{remark}\label{remark:difference of def of zetaGZ and zetaMP}
We recall that the renormalized values (denoted by $\zeta_{\GZ}(k_1,\dots,k_r)$) in \cite{GZ} are defined for $k_1,\dots,k_r\in\Z_{\leq0}$, and the ones (denoted by $\zeta_{\MP}(k_1,\dots,k_r)$) in \cite{MP} are defined on $k_1,\dots,k_r\in\Z$. 
Hence, $\zeta_{\MP}(k_1,\dots,k_r)$ can be regarded as an element of $X_{\C,\zeta^*}$ but it is not clear whether there is an element $\phi$ of $X_{\C,\zeta^*}$ such that
$$
\phi(z_{k_1}\cdots z_{k_r})=\zeta_{\GZ}(k_1, \dots, k_r),
$$
for $k_1,\dots,k_r\in\Z_{\leq0}$.
In any case, we have elements $\mathfrak z_{\GZ}$ and $\mathfrak z_{\MP}$ of $X_{\leq0}$ which satisfy
$$
\mathfrak z_{\GZ}(z_{k_1}\cdots z_{k_r})=\zeta_{\GZ}(k_1,\dots,k_r),
\qquad
\mathfrak z_{\MP}(z_{k_1}\cdots z_{k_r})=\zeta_{\MP}(k_1,\dots,k_r),
$$
for $k_1,\dots,k_r\in\Z_{\leq0}$.
\end{remark}
\section{Explicit relationship}\label{sub:Relationship between renormalized values}
\indent


In this section, we settle Problem \ref{prob: Singer problem} in Theorem \ref{thm:EMS=GZ,MP}.
From now on, we assume that $\mathfrak z$ is an element of $X_{\leq0}$, that is, $\mathfrak z$ is an algebra homomorphism from $\mathcal H_{\leq0}$ to $\C$ and $\mathfrak z$ satisfies
\begin{align}
\mathfrak z |_{N \cap \mathcal H_{\leq0}}=\zeta^*.
\end{align}
By extension of scalars $\C [[t_1,\dots,t_r]]\otimes_\C \mathcal H=\mathcal H_{\leq0} [[t_1,\dots,t_r]]$, we sometimes regard $\mathfrak z$ as a map from $\mathcal H_{\leq0} [[t_1,\dots,t_r]]$ to $\C [[t_1,\dots,t_r]]$.
\begin{remark}\label{cond: harmonic and Riemann zeta}
Because $z_{-k}$ ($k\geq0$) is an element of the vector space $N$ (introduced in Definition \ref{def: non-singular}), we have
$$
\mathfrak z(z_{-k})=\zeta(-k),
$$
for $k\geq0$.
\end{remark}

Let $T:=\{t_i\}_{i\in\N}$. We put $T_{\Z}$ to be the free $\Z$-module generated by all elements of $T$, that is, $T_{\Z}$ is defined by
$$
T_{\Z}:=\left\{\sum_{i=1}^na_it_i \ \middle|\ n\in\N, a_i\in\Z \right\}.
$$
We define $T_{\Z}^\bullet$ to be the non-commutative free monoid generated by all elements of $T_{\Z}$ with the empty word $\emptyset$. We denote each element $\omega=u_1\cdots u_r\in T_\Z^\bullet$ with $u_1,\dots,u_r\in T_\Z$ by $\omega=[u_1,\dots, u_r]$ as a sequence and we denote the concatenation $uv$ with $u,v\in T_\Z^\bullet$ by $[u,v]$.
The length of $\omega=[u_1,\dots, u_r]$ is defined to be $l(\omega)=r$.
We set $\mathcal A_T:=\C\langle T_{\Z}\rangle$ to be the non-commutative polynomial ring generated by $T_{\Z}$.
We define the harmonic product $*:\mathcal A_T^{\otimes 2}\rightarrow\mathcal A_T$ by $\emptyset*w:=w*\emptyset:=w$ and 
\begin{equation}\label{eqn:def of harmonic on AT}
	[u_1,w_1]*[u_2,w_2]:=[u_1,w_1*[u_2,w_2]]+[u_2,[u_1,w_1]*w_2]+[u_1+u_2,w_1*w_2],
\end{equation}
for $w,w_1,w_2\in T_\Z^\bullet$ and $u_1,u_2\in T_\Z$. Then the pair $(\mathcal A_T,*)$ is a commutative, associative, unital $\C$-algebra.
We define
\footnote{The harmonic product is sometimes called the quasi-shuffle product. The symbol ${\rm QSh}$ comes from this name.}
the family $\{\QSh{\omega}{\eta}{\alpha}\}_{\omega,\eta,\alpha\in T_\Z^\bullet}$ in $\Z$ by
$$
\omega * \eta = \sum_{\alpha\in T_\Z^\bullet}\QSh{\omega}{\eta}{\alpha}\alpha.
$$
\begin{example}\label{ex:harmonic 1*r}
For $r\geq1$, we have
\begin{align*}
	[t_{r+1}]*[t_1,\dots,t_r]
	&=\sum_{j=1}^{r+1}[t_1,\dots,t_{j-1},t_{r+1},t_{j},\dots,t_r] \\
	&\quad+\sum_{j=1}^r[t_1,\dots,t_{j-1},t_{r+1}+t_j,t_{j+1},\dots,t_r].
\end{align*}
\end{example}

\begin{definition}
For $r\geq1$, we define the generating functions $Z_*(t_1,\dots,t_r)\in\C[[t_1,\dots,t_r]]$ of the family $\{ \mathfrak z(z_{k_1}\cdots z_{k_r})\in\C\ |\ k_1,\dots,k_r\in\Z_{\leq0} \}$ by
\begin{align}\label{eqn:gen fun of ren val of har type}
Z_*(t_1,\dots,t_r)
:=\sum_{k_1,\dots,k_r\geq0}\frac{(-t_1)^{k_1}\cdots(-t_r)^{k_r}}{k_1!\cdots k_r!}\mathfrak z(z_{-k_1}\cdots z_{-k_r}).
\end{align}
\end{definition}

We put $g:\mathcal A_T\rightarrow \cup_{r\geq1}\C[[t_1,\dots,t_r]]$ to be the $\C$-linear map defined by $g(\emptyset):=1$ and
\begin{equation*}
	g\bigl([u_1,\dots,u_r]\bigr):=Z_*(u_1,\dots,u_r),
\end{equation*}
for $r\geq1$ and $u_1,\dots,u_r\in T_\Z$.
Then the following lemma holds.
\begin{lem}
The map $g$ is an algebra homomorphism, that is, we have
\begin{equation}\label{eqn:alg. hom. of fQ}
	g\bigl( \omega*\eta \bigr)
	=g\bigl( \omega \bigr)g\bigl( \eta \bigr)
\end{equation}
for any $\omega,\eta\in T_\Z^\bullet$.
\end{lem}
\begin{proof}
For $r\geq1$, we put
\begin{align*}
	\widetilde{Z_*}(t_1,\dots,t_r)
	:=\sum_{k_1,\dots,k_r\geq0}
	\frac{(-t_1)^{k_1}\cdots (-t_r)^{k_r}}{k_1!\cdots k_r!}z_{-k_1}\cdots z_{-k_r}
	\in\mathcal H_{\leq0} [[t_1,\dots,t_r]].
\end{align*}
Because we have $\mathfrak z\left(\widetilde{Z_*}(t_1,\dots,t_r)\right)=Z_*(t_1,\dots,t_r)$ and $\mathfrak z$ is an algebra homomorphism, we have
\begin{align*}
	\mathfrak z\left(\widetilde{Z_*}(t_1,\dots,t_r)*\widetilde{Z_*}(t_{r+1},\dots,t_{r+s})\right)
	&=Z_*(t_1,\dots,t_r)Z_*(t_{r+1},\dots,t_{r+s}) \\
	&=g\bigl( [t_1,\cdots,t_r] \bigr)g\bigl( [t_{r+1},\cdots,t_{r+s}] \bigr),
\end{align*}
for $r,s\geq1$.
Hence, it is sufficient to prove 
\begin{align}\label{eqn: alg. hom. of fQ}
	\widetilde{Z_*}(t_1,\dots,t_r)*\widetilde{Z_*}(t_{r+1},\dots,t_{r+s})
	&= \sum_{\alpha\in T_\Z^\bullet}\QSh{[t_1,\cdots,t_r]}{[t_{r+1},\cdots,t_{r+s}]}{\alpha}\widetilde{Z_*}(\alpha),
\end{align}
for $r,s\geq1$.
We have
\begin{align*}
	&\widetilde{Z_*}(t_1,\dots,t_r)*\widetilde{Z_*}(t_{r+1},\dots,t_{r+s}) \\
	=&\sum_{k_1,\dots,k_{r+s}\geq0}
	\frac{(-t_1)^{k_1}\cdots (-t_{r+s})^{k_{r+s}}}{k_1!\cdots k_{r+s}!}
	(z_{-k_1}\cdots z_{-k_r}*z_{-k_{r+1}}\cdots z_{-k_{r+s}}).
\intertext{Here, by definition \eqref{eqn:def of harmonic}, we calculate}
	=&\sum_{k_1,\dots,k_{r+s}\geq0}
	\left\{z_{-k_1}(z_{-k_2}\cdots z_{-k_r}*z_{-k_{r+1}}\cdots z_{-k_{r+s}})
	+z_{-k_{r+1}}(z_{-k_1}\cdots z_{-k_r}*z_{-k_{r+2}}\cdots z_{-k_{r+s}})\right. \\
	&\hspace{3cm}\left.+z_{-k_1-k_{r+1}}(z_{-k_2}\cdots z_{-k_r}*z_{-k_{r+2}}\cdots z_{-k_{r+s}})\right\}
	\frac{(-t_1)^{k_1}\cdots (-t_{r+s})^{k_{r+s}}}{k_1!\cdots k_{r+s}!} \\
	=&	\widetilde{Z_*}(t_1)\left\{ \widetilde{Z_*}(t_2,\cdots,t_r)*\widetilde{Z_*}(t_{r+1},\cdots,t_{r+s})\right\} 
	+\widetilde{Z_*}(t_{r+1})\left\{ \widetilde{Z_*}(t_1,\cdots,t_r)*\widetilde{Z_*}(t_{r+2},\cdots,t_{r+s})\right\} \\
	&\quad + \left( \sum_{k_1,k_{r+1}\geq0}
	\frac{(-t_1)^{k_1}(-t_{r+1})^{k_{r+1}}}{k_1!k_{r+1}!}z_{-k_1-k_{r+1}} \right)
	\left\{\widetilde{Z_*}(t_2,\cdots,t_{r})*\widetilde{Z_*}(t_{r+2},\cdots,t_{r+s})\right\} \\
	=&	\widetilde{Z_*}(t_1)\left\{ \widetilde{Z_*}(t_2,\cdots,t_r)*\widetilde{Z_*}(t_{r+1},\cdots,t_{r+s})\right\} 
	+\widetilde{Z_*}(t_{r+1})\left\{ \widetilde{Z_*}(t_1,\cdots,t_r)*\widetilde{Z_*}(t_{r+2},\cdots,t_{r+s})\right\} \\
	&\quad + \widetilde{Z_*}(t_1 + t_{r+1})
	\left\{\widetilde{Z_*}(t_2,\cdots,t_{r})*\widetilde{Z_*}(t_{r+2},\cdots,t_{r+s})\right\}.
\intertext{By induction hypothesis, we get}
	=&	\widetilde{Z_*}(t_1)\sum_{\alpha\in T_\Z^\bullet}\QSh{[t_2,\cdots,t_r]}{[t_{r+1},\cdots,t_{r+s}]}{\alpha}\widetilde{Z_*}(\alpha) \\
	& + \widetilde{Z_*}(t_{r+1})\sum_{\alpha\in T_\Z^\bullet}\QSh{[t_1,\cdots,t_r]}{[t_{r+2},\cdots,t_{r+s}]}{\alpha}\widetilde{Z_*}(\alpha) \\
	&\quad + \widetilde{Z_*}(t_1 + t_{r+1})
	\sum_{\alpha\in T_\Z^\bullet}\QSh{[t_2,\cdots,t_{r}]}{[t_{r+2},\cdots,t_{r+s}]}{\alpha}
	\widetilde{Z_*}(\alpha).
\intertext{Here, by the definition of $\widetilde{Z_*}$, we see that $\widetilde{Z_*}(t)\widetilde{Z_*}(\alpha)=\widetilde{Z_*}(t,\alpha)$ holds for $t\in T_\Z$ and $\alpha\in T_\Z^\bullet$.
Therefore, we have}
	=&	\sum_{\alpha\in T_\Z^\bullet}\QSh{[t_2,\cdots,t_r]}{[t_{r+1},\cdots,t_{r+s}]}{\alpha}\widetilde{Z_*}([t_1,\alpha]) \\
	& + \sum_{\alpha\in T_\Z^\bullet}\QSh{[t_1,\cdots,t_r]}{[t_{r+2},\cdots,t_{r+s}]}{\alpha}\widetilde{Z_*}([t_{r+1},\alpha]) \\
	&\quad + \sum_{\alpha\in T_\Z^\bullet}\QSh{[t_2,\cdots,t_{r}]}{[t_{r+2},\cdots,t_{r+s}]}{\alpha}
	\widetilde{Z_*}([t_1 + t_{r+1},\alpha]).
\intertext{By using the definition \eqref{eqn:def of harmonic on AT}, we get}
	=&	\sum_{\alpha\in T_\Z^\bullet}
	\QSh{[t_1,\cdots,t_r]}{[t_{r+1},\cdots,t_{r+s}]}{\alpha}\widetilde{Z_*}(\alpha).
\end{align*}
Hence, we obtain \eqref{eqn: alg. hom. of fQ}.
\end{proof}
In order to prove Proposition \ref{prop:r-product of gen fun of GZ}, we prepare Lemma \ref{lem:permutation}.
For $r,i\in\N$ with $i\leq r$, we define $\mathcal P(r,i)$ to be the set of all surjective maps from $\{1,\dots,r\}$ to $\{1,\dots,i\}$.
For any element $\sigma\in\mathcal P(r,i)$ and $1\leq k\leq i$, we put
$$
t_{\sigma^{-1}(k)}:=\sum_{n\in\sigma^{-1}(k)}t_n.
$$
We note that $\mathcal P(r,r)$ is equal to the symmetric group of degree $r$, and we note that $\#\mathcal P(r,1)=1$, that is, the only element $\sigma \in\mathcal P(r,1)$ is given by $\sigma(k):=1$ for $1\leq k\leq r$.
\footnote{In example \ref{ex: thm in sec4}, we explicitly compute $\mathcal P(r,i)$ for $r=2,3$.}
\begin{lem}\label{lem:permutation}
Let $r\geq1$. Then, for $1\leq i\leq r+1$, the summation
\begin{align*}
	\sum_{\sigma\in\mathcal P(r+1,i)}[t_{\sigma^{-1}(1)},\dots,t_{\sigma^{-1}(i)}]
\end{align*}
is equal to
\begin{align*}
	&\sum_{j=1}^i
	\left\{ \sum_{\tau\in\mathcal P(r,i-1)}
	[t_{\tau^{-1}(1)},\dots,t_{\tau^{-1}(j-1)},t_{r+1},t_{\tau^{-1}(j)},\dots,t_{\tau^{-1}(i-1)}] \right.  \\
	&\quad+\left. \sum_{\tau\in\mathcal P(r,i)}
	[t_{\tau^{-1}(1)},\dots,t_{\tau^{-1}(j-1)},t_{r+1}+t_{\tau^{-1}(j)},t_{\tau^{-1}(j+1)},\dots,t_{\tau^{-1}(i)}] \right\}.
\end{align*}
Here, for $i=0$ and $r+1$, we put $\mathcal P(r,i)$ to be the empty set.
\end{lem}
\begin{proof}
When $i=1$, we have
\begin{align*}
	\sum_{\sigma\in\mathcal P(r+1,1)}[t_{\sigma^{-1}(1)}]
	=[t_1+\cdots+t_{r+1}]
	=\sum_{\tau\in\mathcal P(r,1)}[t_{r+1}+t_{\tau^{-1}(1)}].
\end{align*}
Hence, we get the claim for $i=1$.
When $2\leq i\leq r$, take an element $\sigma\in\mathcal P(r+1,i)$.
Then there uniquely exists $j\in\{1,\dots, i\}$ such that $\sigma(r+1)=j$.
If $\#\sigma^{-1}(j)=1$, there uniquely exists $\tau\in\mathcal P(r,i-1)$ which satisfies
$$
\sigma^{-1}(k)
=\left\{\begin{array}{ll}
	\tau^{-1}(k) & (1\leq k\leq j-1), \\
	\tau^{-1}(k-1) & (j\leq k\leq i).
\end{array}\right.
$$
On the other hand, if $\#\sigma^{-1}(j)\geq2$, there uniquely exists $\tau\in\mathcal P(r,i)$ which satisfies
$$
\sigma^{-1}(k)
=\left\{\begin{array}{ll}
	\tau^{-1}(k)\cup\{r+1\} & (k=j), \\
	\tau^{-1}(k) & (k\neq j).
\end{array}\right.
$$
Hence, we get the claim for $2\leq i\leq r$.
When $i=r+1$, take an element $\sigma\in\mathcal P(r+1,r+1)=\mathfrak S_{r+1}$.
Then there uniquely exists $j\in\{1,\dots, r+1\}$ such that $\sigma(r+1)=j$, and for any $1\leq k\leq r+1$, we have $\#\sigma^{-1}(k)=1$, that is, we get
$$
\left\{\sigma^{-1}(1),\dots,\sigma^{-1}(j-1) ,\sigma^{-1}(j+1) ,\dots,\sigma^{-1}(r+1)\right\}
=\{1,\dots,r\}.
$$
So there uniquely exists $\tau\in\mathfrak S_r=\mathcal P(r,r)$ such that
$$
\left(\sigma^{-1}(1),\dots,\sigma^{-1}(j-1) ,\sigma^{-1}(j+1) ,\dots,\sigma^{-1}(r+1)\right)
=\left(\tau^{-1}(1),\dots,\tau^{-1}(r)\right).
$$
Therefore, we get the claim for $i=r+1$.
Hence, we finish the proof.
\end{proof}
\begin{prop}[cf. {\cite[Lemma 5.2; $q=1$]{Hof3}}]\label{prop:r-product of gen fun of GZ}
For $r\geq1$, we have
\begin{equation}\label{eqn:r-product of gen fun of GZ}
	Z_*(t_1)\cdots Z_*(t_r)
	=\sum_{i=1}^r\sum_{\sigma\in\mathcal P(r,i)}
	Z_*\left(t_{\sigma^{-1}(1)},\dots,t_{\sigma^{-1}(i)}\right).
\end{equation}
\end{prop}
\begin{proof}
We prove this claim by induction on $r$. When $r=1$, the element $\sigma\in\mathcal P(1,1)$ is only the identity map, i.e., $\sigma^{-1}(1)=\{1\}$. Hence, the right hand side of \eqref{eqn:r-product of gen fun of GZ} is equal to $Z_*(t_1)$. Assume that the equation \eqref{eqn:r-product of gen fun of GZ} holds for $r=r_0\geq1$. When $r=r_0+1$, by multiplying $Z_*(t_{r_0+1})$ to the both sides of \eqref{eqn:r-product of gen fun of GZ} for $r=r_0$, we have
\begin{align*}
	&Z_*(t_{r_0+1})\sum_{i=1}^{r_0}\sum_{\sigma\in\mathcal P(r_0,i)}
	Z_*\left(t_{\sigma^{-1}(1)},\dots,t_{\sigma^{-1}(i)}\right) \\
	&=\sum_{i=1}^{r_0}\sum_{\sigma\in\mathcal P(r_0,i)}
	g\left([t_{r_0+1}]*[t_{\sigma^{-1}(1)},\dots,t_{\sigma^{-1}(i)}]\right).
\intertext{By Example \ref{ex:harmonic 1*r}, we calculate}
	&=\sum_{i=1}^{r_0}\sum_{\sigma\in\mathcal P(r_0,i)}
	g\left(\sum_{j=1}^{i+1}[t_{\sigma^{-1}(1)},\dots,t_{\sigma^{-1}(j-1)},t_{r_0+1},t_{\sigma^{-1}(j)},\dots,t_{\sigma^{-1}(i)}] \right. \\
	&\quad \left.+\sum_{j=1}^i[t_{\sigma^{-1}(1)},\dots,t_{\sigma^{-1}(j-1)},t_{r_0+1}+t_{\sigma^{-1}(j)},t_{\sigma^{-1}(j+1)},\dots,t_{\sigma^{-1}(i)}] \right).
\intertext{By decomposing each summations, we have}
	&=\sum_{\sigma\in\mathcal P(r_0,r_0)}
	g\left(\sum_{j=1}^{r_0+1}[t_{\sigma^{-1}(1)},\dots,t_{\sigma^{-1}(j-1)},t_{r_0+1},t_{\sigma^{-1}(j)},\dots,t_{\sigma^{-1}(r_0)}] \right) \\
	&\quad+\sum_{i=1}^{r_0-1}\sum_{\sigma\in\mathcal P(r_0,i)}
	g\left(\sum_{j=1}^{i+1}[t_{\sigma^{-1}(1)},\dots,t_{\sigma^{-1}(j-1)},t_{r_0+1},t_{\sigma^{-1}(j)},\dots,t_{\sigma^{-1}(i)}] \right) \\
	&\quad +\sum_{i=2}^{r_0}\sum_{\sigma\in\mathcal P(r_0,i)}g\left(\sum_{j=1}^i[t_{\sigma^{-1}(1)},\dots,t_{\sigma^{-1}(j-1)},t_{r_0+1}+t_{\sigma^{-1}(j)},t_{\sigma^{-1}(j+1)},\dots,t_{\sigma^{-1}(i)}] \right) \\
	&\quad +\sum_{\sigma\in\mathcal P(r_0,1)}g\left([t_{r_0+1}+t_{\sigma^{-1}(1)}] \right).
\end{align*}
By applying Lemma \ref{lem:permutation} for $r=r_0$ and $i=r_0+1$ (resp. $i=1$) to the first term (resp. the fourth term), we get
\begin{align*}
	&=
	g\left(\sum_{\sigma\in\mathcal P(r_0+1,r_0+1)}[t_{\sigma^{-1}(1)},\dots,t_{\sigma^{-1}(r_0+1)}] \right) \\
	&\quad+\sum_{i=2}^{r_0}g\left(\sum_{j=1}^{i}\left\{\sum_{\sigma\in\mathcal P(r_0,i-1)}
	[t_{\sigma^{-1}(1)},\dots,t_{\sigma^{-1}(j-1)},t_{r_0+1},t_{\sigma^{-1}(j)},\dots,t_{\sigma^{-1}(i-1)}] \right.\right. \\
	&\quad +\left.\left.\sum_{\sigma\in\mathcal P(r_0,i)}[t_{\sigma^{-1}(1)},\dots,t_{\sigma^{-1}(j-1)},t_{r_0+1}+t_{\sigma^{-1}(j)},t_{\sigma^{-1}(j+1)},\dots,t_{\sigma^{-1}(i)}] \right\}\right) \\
	&\quad +g\left(\sum_{\sigma\in\mathcal P(r+1,1)}[t_{\sigma^{-1}(1)}] \right).
\end{align*}
By applying Lemma \ref{lem:permutation} for $r=r_0$ and $2\leq i\leq r_0$ to the second term, we get
\begin{align*}
	Z_*(t_{r_0+1})\sum_{i=1}^{r_0}\sum_{\sigma\in\mathcal P(r_0,i)}
	Z_*\left(t_{\sigma^{-1}(1)},\dots,t_{\sigma^{-1}(i)}\right)
	&=\sum_{i=1}^{r_0+1}g\left(\sum_{\sigma\in\mathcal P(r_0+1,i)}[t_{\sigma^{-1}(1)},\dots,t_{\sigma^{-1}(i)}]\right) \\
	&=\sum_{i=1}^{r_0+1}\sum_{\sigma\in\mathcal P(r_0+1,i)}
	Z_*\left(t_{\sigma^{-1}(1)},\dots,t_{\sigma^{-1}(i)}\right).
\end{align*}
Hence, we obtain the claim.
\end{proof}
By using the above proposition, we get a universal presentation of $Z_{\EMS}(t_1,\dots,t_r)$ (defined by \eqref{eqn:gen fun of zetaEMS}) by any generating functions of renormalized values of harmonic type.
\begin{thm}\label{thm:EMS=GZ,MP}
Let $\mathfrak z$ be a renormalized values of harmonic type (cf. Definition \ref{def:ren val of har type}), and let $Z_*$ be the generating function of $\mathfrak z$ given by \eqref{eqn:gen fun of ren val of har type}. Then for $r\geq1$, we have
\begin{equation}\label{eqn:EMS=GZ,MP}
Z_{\EMS}(t_1,\dots,t_r)
	=\sum_{i=1}^r\sum_{\sigma\in\mathcal P(r,i)}
	Z_*\left(u_{\sigma^{-1}(1)},\dots,u_{\sigma^{-1}(i)}\right).
\end{equation}
Here, $u_{\sigma^{-1}(k)}$ is defined by
$$
u_{\sigma^{-1}(k)}:=\sum_{n\in\sigma^{-1}(k)}u_n,
$$
for $u_i:=t_i+\cdots+t_r$ ($1\leq i\leq r$).
\end{thm}
\begin{proof}
By Remark \ref{cond: harmonic and Riemann zeta}, we have
$$
Z_{\EMS}(t_1)=Z_*(t_1).
$$
Therefore, by the equation \eqref{eqn:recurrence formula of gen fun}, we have
\begin{align*}
	Z_{\EMS}(t_1,\dots,t_r)
	=\prod_{i=1}^rZ_{\EMS}(t_i+\cdots+t_r)=\prod_{i=1}^rZ_*(t_i+\cdots+t_r).
\end{align*}
Therefore, by putting $u_i:=t_i+\cdots+t_r$ ($1\leq i\leq r$) and by using the equation \eqref{eqn:r-product of gen fun of GZ}, we obtain
\begin{equation*}
	Z_{\EMS}(t_1,\dots,t_r)
	=\sum_{i=1}^r\sum_{\sigma\in\mathcal P(r,i)}
	Z_*\left(u_{\sigma^{-1}(1)},\dots,u_{\sigma^{-1}(i)}\right).
\end{equation*}
Hence, we finish the proof.
\end{proof}

In the following example, we denote $\sigma\in\mathcal P(r,i)$ by
$$
\left(\begin{array}{ccc}
	1 & \cdots & r \\
	\sigma(1) & \cdots & \sigma(r)
\end{array}\right).
$$

\begin{example}\label{ex: thm in sec4}
When $r=2$, we have
$$
\mathcal P(2,2)
=\left\{
\left(\begin{array}{cc}
	1 & 2 \\
	1 & 2
\end{array}\right),
\left(\begin{array}{cc}
	1 & 2 \\
	2 & 1
\end{array}\right)
\right\},\quad
\mathcal P(2,1)
=\left\{
\left(\begin{array}{cc}
	1 & 2 \\
	1 & 1
\end{array}\right)
\right\},
$$
so we get
\begin{equation*}
	Z_{\EMS}(t_1,t_2)
	=Z_*(t_1+t_2,t_2)+Z_*(t_2,t_1+t_2)+Z_*(t_1+2t_2).
\end{equation*}
When $r=3$, we have $\mathcal P(3,1)=\scalebox{.6}{$\left\{
\left(\begin{array}{ccc}
	1 & 2 & 3 \\
	1 & 1 & 1
\end{array}\right)
\right\}$},$
and $\mathcal P(3,2)$ is given by
$$\scalebox{.8}{
$\left\{
\left(\begin{array}{ccc}
	1 & 2 & 3 \\
	1 & 1 & 2
\end{array}\right),
\left(\begin{array}{ccc}
	1 & 2 & 3 \\
	1 & 2 & 1
\end{array}\right),
\left(\begin{array}{ccc}
	1 & 2 & 3 \\
	2 & 1 & 1
\end{array}\right),
\left(\begin{array}{ccc}
	1 & 2 & 3 \\
	2 & 2 & 1
\end{array}\right),
\left(\begin{array}{ccc}
	1 & 2 & 3 \\
	2 & 1 & 2
\end{array}\right),
\left(\begin{array}{ccc}
	1 & 2 & 3 \\
	1 & 2 & 2
\end{array}\right)
\right\}$
},$$
and $\mathcal P(3,3)$ is given by
$$\scalebox{.8}{
$\left\{
\left(\begin{array}{ccc}
	1 & 2 & 3 \\
	1 & 2 & 3
\end{array}\right),
\left(\begin{array}{ccc}
	1 & 2 & 3 \\
	1 & 3 & 2
\end{array}\right),
\left(\begin{array}{ccc}
	1 & 2 & 3 \\
	2 & 1 & 3
\end{array}\right),
\left(\begin{array}{ccc}
	1 & 2 & 3 \\
	2 & 3 & 1
\end{array}\right),
\left(\begin{array}{ccc}
	1 & 2 & 3 \\
	3 & 1 & 2
\end{array}\right),
\left(\begin{array}{ccc}
	1 & 2 & 3 \\
	3 & 2 & 1
\end{array}\right)
\right\}$
}.$$
Hence, we get
\begin{align*}
	Z_{\EMS}(t_1,t_2,t_3)
	& =Z_*(t_1+t_2+t_3,t_2+t_3,t_3) + Z_*(t_1+t_2+t_3,t_3,t_2+t_3) \\
	& \quad+ Z_*(t_2+t_3,t_1+t_2+t_3,t_3) + Z_*(t_2+t_3,t_3,t_1+t_2+t_3) \\
	& \quad+ Z_*(t_3,t_1+t_2+t_3,t_2+t_3) + Z_*(t_3,t_2+t_3,t_1+t_2+t_3) \\
	& \quad+ Z_*(t_1+t_2+2t_3,t_2+t_3)+Z_*(t_2+t_3,t_1+t_2+2t_3) \\
	& \quad+ Z_*(t_2+2t_3,t_1+t_2+t_3)+Z_*(t_1+t_2+t_3,t_2+2t_3) \\
	& \quad+ Z_*(t_1+2t_2+2t_3,t_3)+Z_*(t_3,t_1+2t_2+2t_3) \\
	& \quad+ Z_*(t_1+2t_2+3t_3).
\end{align*}
\end{example}

\begin{cor}\label{cor:EMS=GZ,MP}
The equation \eqref{eqn:EMS=GZ,MP} holds for $Z_*=Z_{\GZ}$ and $Z_{\MP}$ defined by
\begin{align*}
	&Z_{\GZ}(t_1,\dots,t_r):=\sum_{k_1,\dots,k_r=0}^{\infty}\frac{(-t_1)^{k_1}\cdots(-t_r)^{k_r}}{k_1!\cdots k_r!}\zeta_{\GZ}(-k_1,\dots,-k_r), \\
	&Z_{\MP}(t_1,\dots,t_r):=\sum_{k_1,\dots,k_r=0}^{\infty}\frac{(-t_1)^{k_1}\cdots(-t_r)^{k_r}}{k_1!\cdots k_r!}\zeta_{\MP}(-k_1,\dots,-k_r).
\end{align*}
Hence, the renormalized values $\zeta_{\EMS}(-k_1,\dots,-k_r)$ can be represented by a finite linear combination of either $\zeta_{\GZ}(-k_1,\dots,-k_r)$ or $\zeta_{\GZ}(-k_1,\dots,-k_r)$.
\end{cor}
\begin{proof}
We recall elements $\mathfrak z_{\GZ}$ and $\mathfrak z_{\MP}$ of $X_{\leq0}$ in Remark \ref{remark:difference of def of zetaGZ and zetaMP}. These elements satisfy
$$
\mathfrak z_{\GZ}(z_{k_1}\cdots z_{k_r})=\zeta_{\GZ}(k_1,\dots,k_r),
\qquad
\mathfrak z_{\MP}(z_{k_1}\cdots z_{k_r})=\zeta_{\MP}(k_1,\dots,k_r),
$$
for $k_1,\dots,k_r\in\Z_{\leq0}$.
Hence, we get the claim.
\end{proof}


\end{document}